\documentclass[graybox]{svmult}

\smartqed
\usepackage{mathptmx}       
\usepackage{helvet}         
\usepackage{courier}        
\usepackage{type1cm}        
\usepackage{graphicx}       

\usepackage{array,colortbl}
\usepackage{amsmath,amsfonts,amssymb,bm} 
\DeclareFontFamily{U}{mathx}{\hyphenchar\font45}
\DeclareFontShape{U}{mathx}{m}{n}{
      <5> <6> <7> <8> <9> <10>
      <10.95> <12> <14.4> <17.28> <20.74> <24.88>
      mathx10
      }{}
\DeclareSymbolFont{mathx}{U}{mathx}{m}{n}
\DeclareFontSubstitution{U}{mathx}{m}{n}
\DeclareMathAccent{\widecheck}      {0}{mathx}{"71}

\newcommand{\Z}{\mathbb{Z}} 
\newcommand{\R}{\mathbb{R}} 
\newcommand{\N}{\mathbb{N}} 
\newcommand{\F}{\mathbb{F}} 
\newcommand{\bszero}{\boldsymbol{0}} 
\newcommand{\bsone}{\boldsymbol{1}}  
\newcommand{\bst}{\boldsymbol{t}}    
\newcommand{\bsx}{\boldsymbol{x}}    
\newcommand{\bsy}{\boldsymbol{y}}    
\newcommand{\bsz}{\boldsymbol{z}}    
\newcommand{\bsDelta}{\boldsymbol{\Delta}}    

\DeclareSymbolFont{bbold}{U}{bbold}{m}{n}
\DeclareSymbolFontAlphabet{\mathbbold}{bbold}
\newcommand{\ind}{\mathbbold{1}}

\usepackage{microtype} 
\usepackage{xspace}
\usepackage[colorlinks=true,linkcolor=black,citecolor=black,urlcolor=black]{hyperref}
\usepackage{bookmark}
\makeatletter 
  \providecommand*{\toclevel@author}{999}
  \providecommand*{\toclevel@title}{0}
\makeatother

\usepackage{mathtools}

\begin{document}
\spnewtheorem{algo}{Algorithm}{\bf}{\rm}
\newcommand{\bsa}{\boldsymbol{a}}    
\newcommand{\bsh}{\boldsymbol{h}}    
\newcommand{\bsi}{\boldsymbol{i}}    
\newcommand{\bsj}{\boldsymbol{j}}    
\newcommand{\bsk}{\boldsymbol{k}}    
\newcommand{\bsl}{\boldsymbol{l}}    
\newcommand{\bsr}{\boldsymbol{r}}    
\newcommand{\bsnu}{\boldsymbol{\nu}}    
\newcommand{\me}{\text{e}}			
\newcommand{\cc}{\mathcal{C}}
\newcommand{\cm}{\mathcal{M}}		
\newcommand{\cl}{\mathcal{L}}
\newcommand{\cn}{\mathcal{N}}
\newcommand{\Order}{\mathcal{O}}
\newcommand{\cp}{\mathcal{P}}
\newcommand{\cx}{\mathcal{X}}
\newcommand{\natm}{\N_{0,m}}
\newcommand{\cube}{[0,1)^d}
\newcommand{\hf}{\hat{f}}
\newcommand{\rf}{\mathring{f}}
\newcommand{\tf}{\tilde{f}}
\newcommand{\hg}{\hat{g}}
\newcommand{\hI}{\hat{I}}
\newcommand{\tvk}{\tilde{\bsk}}
\newcommand{\hS}{\widehat{S}}
\newcommand{\tS}{\widetilde{S}}
\newcommand{\wcS}{\widecheck{S}}
\newcommand{\okappa}{\overline{\kappa}}
\newcommand{\rnu}{\mathring{\nu}}
\newcommand{\tnu}{\widetilde{\nu}}
\newcommand{\hnu}{\widehat{\nu}}
\newcommand{\onu}{\overline{\nu}}
\newcommand{\hbsnu}{\widehat{\bsnu}}   
\newcommand{\homega}{\widehat{\omega}}
\newcommand{\wcomega}{\mathring{\omega}}
\newcommand{\fC}{\mathfrak{C}}
\newcommand{\nodes}{\{\bsz_i\}_{i=0}^{\infty}}
\newcommand{\nodesn}{\{\bsz_i\}_{i=0}^{n-1}}
\newcommand{\norm}[1]{\ensuremath{\left \lVert #1 \right \rVert}}
\newcommand{\abs}[1]{\ensuremath{\left |  #1 \right |}} 
\newcommand{\bigabs}[1]{\ensuremath{\bigl \lvert #1 \bigr \rvert}}
\newcommand{\Bigabs}[1]{\ensuremath{\Bigl \lvert #1 \Bigr \rvert}}
\newcommand{\biggabs}[1]{\ensuremath{\biggl \lvert #1 \biggr \rvert}}
\newcommand{\Biggabs}[1]{\ensuremath{\Biggl \lvert #1 \Biggr \rvert}}
\newcommand{\ip}[3][{}]{\ensuremath{\left \langle #2, #3 \right \rangle_{#1}}}

\newcommand{\lattice}{\cl} 
\newcommand{\Lebesgue}{L} 
\newcommand{\cublat}{\texttt{cubLattice\_g}\xspace}
\allowdisplaybreaks

\title*{Adaptive Multidimensional Integration Based on Rank-1 Lattices}
\author{Llu\'is Antoni Jim\'enez Rugama \and Fred J. Hickernell}
\institute{Llu\'is Antoni \and Fred J. Hickernell \at Department of Applied Mathematics,  Illinois Institute of Technology, 10 W. 32$^{\text{nd}}$ Street, E1-208, Chicago, IL 60616, USA
\email{ljimene1@hawk.iit.edu}, \email{hickernell@iit.edu}}
\maketitle

\abstract{Quasi-Monte Carlo methods are used for numerically integrating multivariate functions. However, the error bounds for these methods typically rely on a priori knowledge of some semi-norm of the integrand, not on the sampled function values. In this article, we propose an error bound based on the discrete Fourier coefficients of the integrand. If these Fourier coefficients decay more quickly, the integrand has less fine scale structure, and the accuracy is higher. We focus on rank-1 lattices because they are a commonly used quasi-Monte Carlo design and because their algebraic structure facilitates an error analysis based on a Fourier decomposition of the integrand. This leads to a guaranteed adaptive cubature algorithm with computational cost $\Order(mb^m)$, where $b$ is some fixed prime number and $b^m$ is the number of data points.}

\section{Introduction}

Quasi-Monte Carlo (QMC) methods use equally weighted sums of integrand values at carefully chosen nodes to approximate multidimensional integrals over the unit cube,
\[
\frac 1n \sum_{i=0}^{n-1} f(\bsz_i)\approx \int_{\cube} f(\bsx) \, \D \bsx.
\]
Integrals over more general domains may often be accommodated by a transformation of the integration variable. QMC methods are widely used because they do not suffer from a \textit{curse of dimensionality}. The existence of QMC methods with dimension-independent error convergence rates is discussed in \cite[Ch.\ 10--12]{NovWoz10a}. See \cite{DicEtal14a} for a recent review.

The QMC convergence rate of $\Order(n^{-(1-\delta)})$ does not give enough information about the absolute error to determine how large $n$ must be to satisfy a given error tolerance, $\varepsilon$. The objective of this research is to develop a guaranteed, QMC algorithm based on rank-1 lattices that determines $n$ adaptively by calculating a data-driven upper bound on the absolute error. The Koksma-Hlawka inequality is impractical for this purpose because it requires the total variation of the integrand. Our data-driven bound is expressed in terms of the integrand's discrete Fourier coefficients. 

Sections \ref{secrank1lat}--\ref{FFT} describe the group structure of rank-1 lattices and how the complex exponential functions are an appropriate basis for these nodes. For computation purposes, there is also an explanation of how to obtain the discrete Fourier transform of $f$ with an $\Order(n \log(n))$ computational cost.  New contributions are described in Section \ref{secerrest} and \ref{algorithmsection}. Initially, a mapping from $\N_0$ to the space of wavenumbers, $\Z^d$, is defined according to constraints given by the structure of our rank-1 lattice node sets. With this mapping, we define a set of integrands for which our new adaptive algorithm is designed.  This set is defined in terms of cone conditions satisfied by the (true) Fourier coefficients of the integrands. These conditions make it possible to derive an upper bound on the rank-1 lattice rule error in terms of the discrete Fourier coefficients, which can be used to construct an adaptive algorithm.  An upper bound on the computational cost of this algorithm is derived. Finally, there is an example of option pricing using the MATLAB implementation of our algorithm, \cublat, which is part of the Guaranteed Automatic Integration Library \cite{ChoEtal15a}. A parallel development for Sobol' cubature is given in \cite{HicJim16a}.

\section{Rank-1 Integration Lattices}\label{secrank1lat}

Let $b$ be prime number, and let $\F_{n}:=\{0, \ldots, n-1\}$ denote the set of the first $n$ non-negative integers for any $n \in \N$. The aim is to construct a sequence of embedded node sets with $b^m$ points for $m \in \N_0$:
\[
\{\bszero\}=:\cp_0\subset \cp_1 \dots\subset\cp_m:=\{\bsz_i\}_{i\in \F_{b^m}} \subset\dots\subset\cp_\infty:=\{\bsz_i\}_{i\in \N_0}.
\]
Specifically, the sequence $\bsz_1, \bsz_b, \bsz_{b^2},  \ldots \in \cube$ is chosen such that 
\begin{subequations} \label{zbmdef}
\begin{gather} 
\bsz_1 = b^{-1} \bsa_0, \qquad \bsa_0 \in \{1, \ldots, b-1\}^{d}, \\
\bsz_{b^m} = b^{-1}(\bsz_{b^{m-1}}+\bsa_m)= b^{-1}\bsa_m + \cdots + b^{-m-1} \bsa_0, \qquad \bsa_m \in \F_b^{d}, \ \ m \in \N.
\end{gather}
\end{subequations}
From this definition it follows that for all $m \in \N_0$,
\begin{equation} 
b^\ell\bsz_{b^m} \bmod {1} = \begin{cases} \bsz_{b^{m-\ell}}, & \ell=0, \ldots, m \\
\bszero, & \ell=m+1, m+2, \ldots .
\end{cases}
\label{latpropb}
\end{equation}
Next, for any $i \in \N$ with proper $b$-ary expansion $i=i_0+i_1 b + i_2 b^2 + \cdots$, and $m=\lfloor \log_b(i) \rfloor+1$ define 
\begin{multline} \label{zidef}
\bsz_i : = \sum_{\ell=0}^{\infty} i_\ell \bsz_{b^\ell} \bmod 1 = \sum_{\ell=0}^{m-1} i_\ell \bsz_{b^\ell} \bmod 1 = \sum_{\ell=0}^{m-1} i_\ell b^{m-1-\ell}  \bsz_{b^{m-1}} \bmod 1 \\ 
 = j \bsz_{b^{m-1}} \bmod 1, \qquad \text{where } j= \sum_{\ell=0}^{m-1} i_\ell b^{m-1-\ell},
\end{multline}
where \eqref{latpropb} was used.  This means that node set $\cp_m$ defined above may be written as the integer multiples of the generating vector $\bsz_{b^{m-1}}$ since 
\begin{multline*}
\cp_m:= \{\bsz_i\}_{i \in \F_{b^m}} = \bigg \{ \bsz_{b^{m-1}} \sum_{\ell=0}^{m-1} i_\ell b^{m-1-\ell} \bmod 1 : i_0, \ldots , i_{m-1} \in \F_b \bigg \} \\
= \left\{ j \bsz_{b^{m-1}} \bmod 1\right \}_{j \in \F_{b^m} }.
\end{multline*}

Integration lattices, $\lattice$, are defined as discrete groups in $\R^d$ containing $\Z^d$ and closed under normal addition \cite[Sec. 2.7-2.8]{SloJoe94}.  The node set of an integration lattice is its intersection with the half-open unit cube, $\cp:=\lattice \cap \cube$. In this case, $\cp$ is also a group, but this time under addition modulo 1, i.e., operator $\oplus:\cube \times \cube \to \cube$ defined by $\bsx\oplus\bsy:=(\bsx+\bsy)\bmod 1$, and where $\ominus \bsx:=\bsone-\bsx$.

Sets $\cp_m$ defined above are embedded node sets of integration lattices. The sufficiency of a single generating vector for each of these $\cp_m$ is the reason that $\cp_m$ is called the node set of a \emph{rank-1} lattice. 
The theoretical properties of good embedded rank-1 lattices for cubature are discussed in \cite{HicNie03a}.  
 
The set of $d$-dimensional integer vectors, $\Z^d$, is used to index Fourier series expressions for the integrands, and $\Z^d$ is also known as the wavenumber space. We define the bilinear operation $\ip{\cdot}{\cdot}: \Z^d \times \cube \to [0,1)$ as the dot product modulo $1$: 
\begin{equation}\label{bilinear}
\ip{\bsk}{\bsx}:=\bsk^T\bsx\bmod 1 \qquad \forall \bsk \in \Z^d, \ \bsx \in \cube.
\end{equation}
This bilinear operation has the following properties: for all $\bst, \bsx \in \cube$, $\bsk, \bsl \in \Z^d$, and $a \in \Z$, it follows that
\begin{subequations}
\begin{gather}
\ip{\bsk}{\bszero} = \ip{\bszero}{\bsx} = 0,\\
\ip{\bsk}{a \bsx \bmod 1 \oplus \bst} = (a\ip{\bsk}{\bsx} + \ip{\bsk}{\bst}) \bmod 1 \label{bilinearlinxprop} \\
\ip{a \bsk + \bsl}{\bsx} = (a\ip{\bsk}{\bsx} + \ip{\bsl}{\bsx}) \bmod 1, \label{bilinearlinkprop}\\
\ip{\bsk}{\bsx} = 0 \ \forall \bsk \in \Z^d \ \implies \ \bsx=\bszero.\label{bilinearlinzeroprop}
\end{gather}
\end{subequations}
An additional constraint placed on the embedded lattices is that
\begin{equation}
\ip{\bsk}{\bsz_{b^m}} =  0 \ \forall m \in \N_0   \ \implies \ \bsk=\bszero. \label{latpropd}
\end{equation}

The bilinear operation defined in \eqref{bilinear} is also used to define the \emph{dual} lattice corresponding to $\cp_m$:
\begin{align}
\nonumber
\cp^{\perp}_m & := \{\bsk \in \Z^d : \ip{\bsk}{\bsz_i} = 0, \ i\in \F_{b^m}\} \\
&= \{\bsk \in \Z^d : \ip{\bsk}{\bsz_{b^{m-1}}} = 0 \} \qquad \text{by \eqref{zidef} and \eqref{bilinearlinxprop}} .\label{dualdef}
\end{align}
By this definition $\cp^{\perp}_{0}=\Z^d$, and the properties \eqref{latpropb}, \eqref{bilinear}, and \eqref{latpropd}, imply also that the $\cp^{\perp}_m$ are nested subgroups with
\begin{equation}\label{dualemb}
\Z^d=\cp^{\perp}_{0}\supseteq\dots\supseteq\cp^{\perp}_{m}\supseteq\dots\supseteq\cp^{\perp}_{\infty}=\{\bszero\}.
\end{equation}
Analogous to the \emph{dual} lattice definition, for $j\in\F_{b^m}$ one can define the \emph{dual} cosets as $\cp^{\perp,j}_m := \{\bsk \in \Z^d : b^m\ip{\bsk}{\bsz_{b^{m-1}}} = j\}$. Hence, a similar extended property \eqref{dualemb} applies:
\begin{equation}\label{dualcosetemb}
\cp^{\perp,j}_{m}=\bigcup_{a=0}^{b-1}\cp^{\perp,j+ab^m}_{m+1}
\Longrightarrow
\cp^{\perp,j}_{m}\supseteq\cp^{\perp,j+ab^m}_{m+1}, \quad a\in\F_{b},\, j\in\F_{b^m}.
\end{equation}
The overall \emph{dual} cosets structure can be represented as a tree, where $\{\cp^{\perp,j+ab^m}_{m+1}\}_{a=0}^{b-1}$ are the children of $\cp^{\perp,j}_{m}$.

Figure \ref{Latticefig} shows an example of a rank-1 lattice node set with $64$ points in dimension $2$ and its dual lattice. The parameters defining this node set are $b=2$, $m=6$, and $\bsz_{32}=(1,27)/64$. It is useful to see how $\cp_{m}=\cp_{m-1}\cup\left\{\cp_{m-1}+\bsz_{2^{m-1}}\bmod 1\right\}$.
\begin{figure}[h!]
\centering
\begin{tabular}{>{\centering}p{5cm}>{\centering}p{5cm}}
\includegraphics[width=4.5cm]{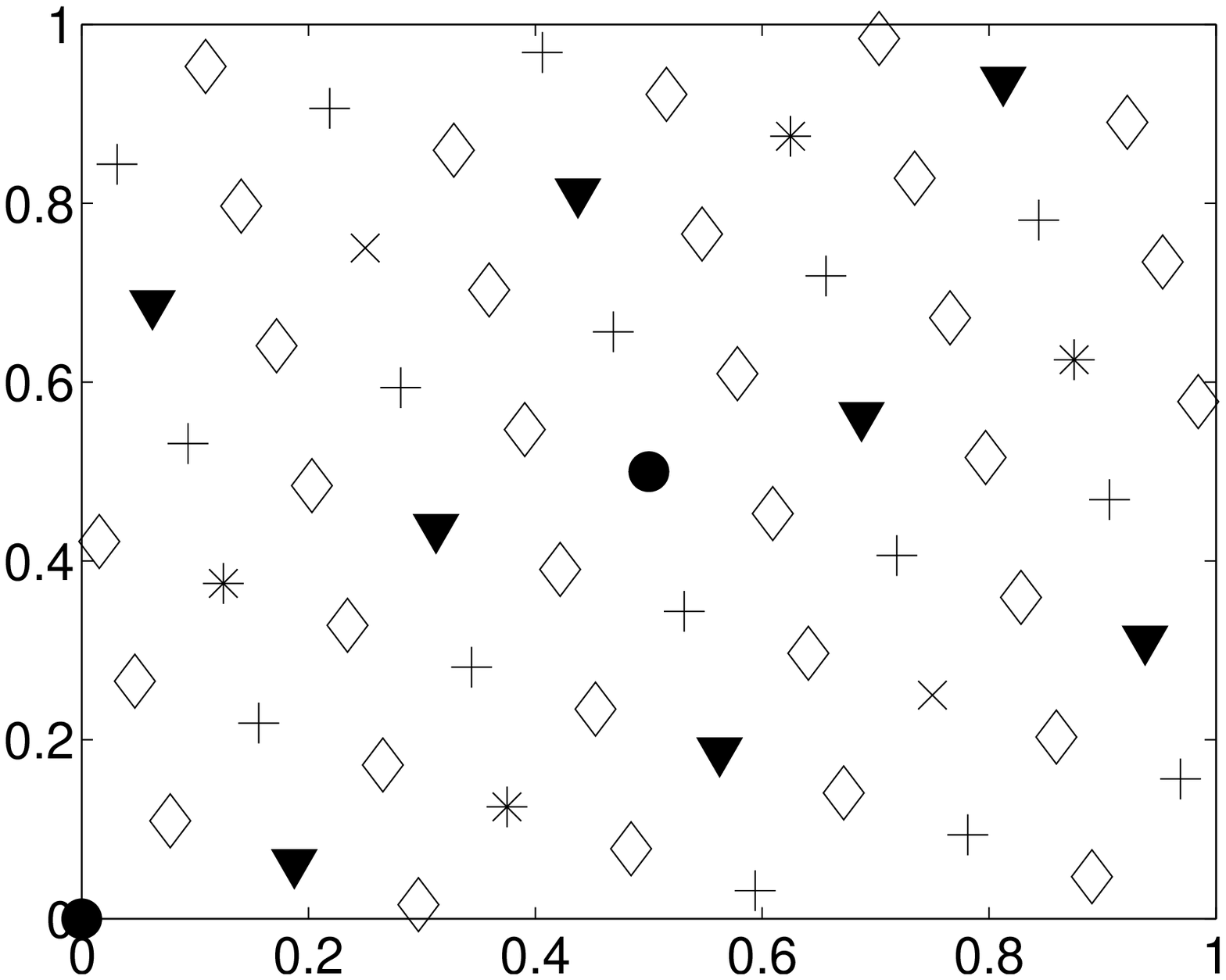} &
\includegraphics[width=4.5cm]{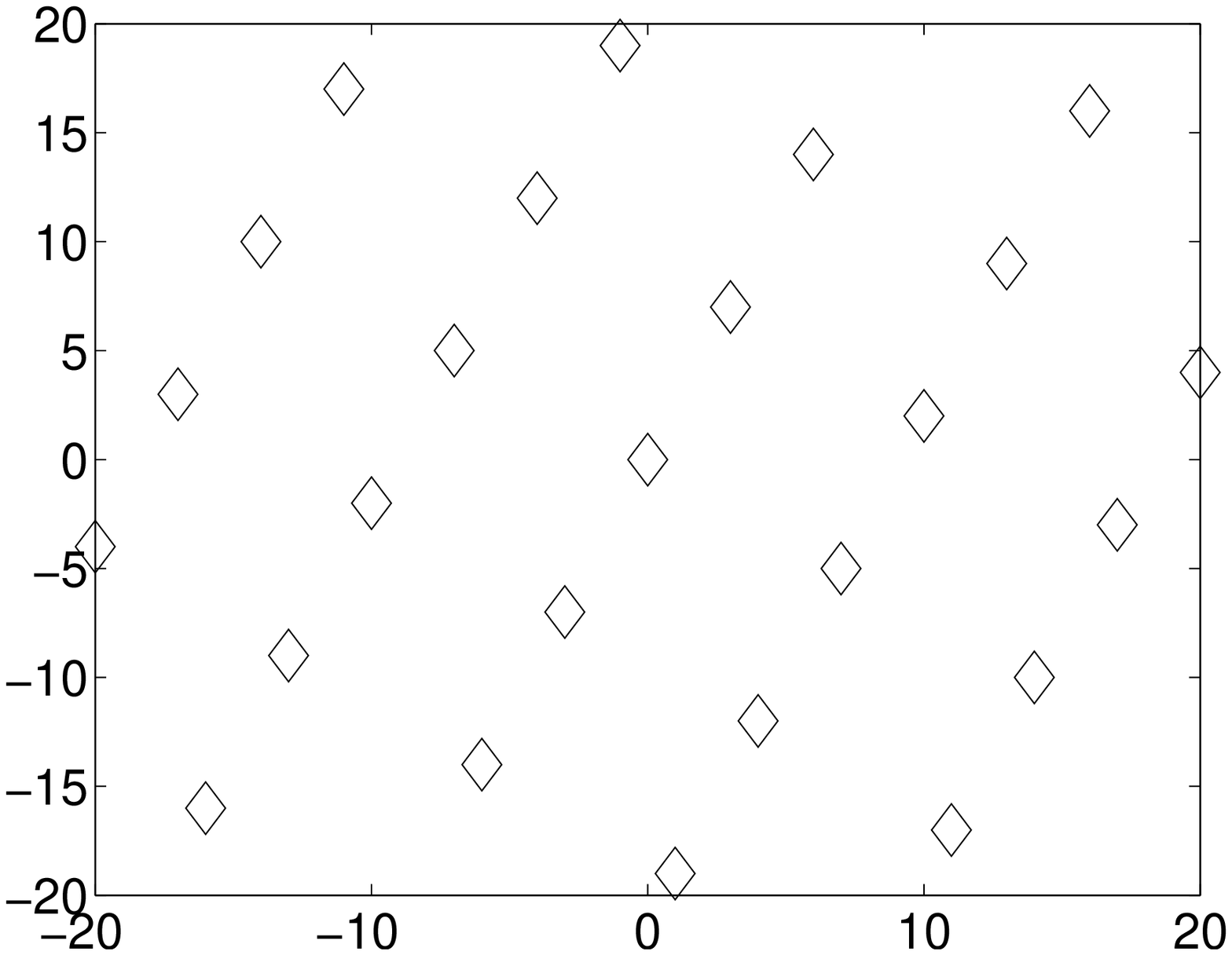}\tabularnewline
a) & b)
\end{tabular}
\caption{Plots of a) the node set $\cp_6$ depicted as $\bullet\{\bsz_{0},\bsz_{1}\}$, $\times\{\bsz_{2},\bsz_{3}\}$, $\ast\{\bsz_{4},\dots,\bsz_{7}\}$, $\blacktriangledown\{\bsz_{8},\dots,\bsz_{15}\}$, $+\{\bsz_{16},\dots,\bsz_{31}\}$, $\diamond\{\bsz_{32},\dots,\bsz_{63}\}$, and b) some of the dual lattice points, $\cp_6^{\perp} \cap [-20,20]^2$.}\label{Latticefig}
\end{figure}

\section{Fourier Series}\label{secfourierseries}

The integrands considered here are absolutely continuous periodic functions. If the integrand is not initially periodic, it may be periodized as discussed in \cite{Hic01a}, \cite{Sid93}, or \cite[Sec. 2.12]{SloJoe94}. More general box domains may be considered, also by using variable transformations, see e.g.,  \cite{HicSloWas03a,HicSloWas03e}.

The $\Lebesgue^2(\cube)$ inner product is defined as
$
\ip[2]{f}{g} = \int_{\cube} f(\bsx) \overline{g(\bsx)} \, \D \bsx.
$
The complex exponential functions, $\{\E^{2 \pi \sqrt{-1} \ip{\bsk}{\cdot}}\}_{\bsk \in \Z^d}$ form a complete orthonormal \emph{basis} for $\Lebesgue^2(\cube)$. So, any function in $\Lebesgue^2(\cube)$ may be written as its Fourier series as
\begin{equation} \label{Fourierdef}
f(\bsx) = \sum_{\bsk \in \Z^d} \hf(\bsk) \E^{2 \pi \sqrt{-1} \ip{\bsk}{\bsx}}, \quad \text{where } \hf(\bsk) = \ip[2]{f}{\E^{2 \pi \sqrt{-1} \ip{\bsk}{\cdot}}},
\end{equation}
and the inner product of two functions in $\Lebesgue^2(\cube)$ is the $\ell^2$ inner product of their series coefficients:
\[
\ip[2]{f}{g} = \sum_{\bsk \in \Z^d} \hf(\bsk)\overline{\hg(\bsk)} =: \ip[2]{\bigl(\hf(\bsk)\bigr)_{\bsk \in \Z^d}}{\bigl ( \hg(\bsk)\bigr )_{\bsk \in \Z^d}}.
\]

Note that for any $\bsz\in\cp_m$ and $\bsk\in\cp_m^\perp$, we have $\E^{2 \pi \sqrt{-1} \ip{\bsk}{\bsz}}=1$. The special group structure of the  lattice node set, $\cp_m$, leads to a useful formula for the average of any Fourier basis function over $\cp_m$. According to \cite[Lemma 5.21]{Nie92},
\begin{equation}\label{avrFourier}
\frac{1}{b^m} \sum_{i=0}^{b^m-1} \E^{2 \pi \sqrt{-1} \ip{\bsk}{\bsz_i}} = \ind_{\cp_m^{\perp}}(\bsk) = \begin{cases} 1 , & \bsk \in \cp_m^{\perp}\\
 0,  & \bsk \in \Z^d \setminus \cp_m^{\perp}.
 \end{cases}
\end{equation}

This property of the dual lattice is used below to describe the absolute error of a shifted rank-1 lattice cubature rule in terms of the Fourier coefficients for wavenumbers in the dual lattice. For fixed $\bsDelta \in \cube$, the cubature rule is defined as
\begin{equation} \label{cubaturedef}
\hI_m(f) := \frac{1}{b^m} \sum_{i=0}^{b^m-1} f(\bsz_i \oplus\bsDelta),  \qquad  m \in \N_0.
\end{equation}
Note from this definition that $\hI_m\left(\E^{2 \pi \sqrt{-1} \ip{\bsk}{\cdot}}\right)= \E^{2 \pi \sqrt{-1} \ip{\bsk}{\bsDelta}}\ind_{\cp_m^{\perp}}(\bsk)$. The series decomposition defined in \eqref{Fourierdef} and equation \eqref{avrFourier} are used in intermediate results from \cite[Theorem 5.23]{Nie92} to show that,
\begin{align}
\biggabs{ \int_{\cube} f(\bsx) \, \D \bsx - \hI_m(f)} 
= \Biggabs {\sum_{\bsk \in \cp_m^{\perp}\setminus \{\bszero\} } \hf(\bsk)  \E^{2 \pi \sqrt{-1} \ip{\bsk}{\bsDelta}} } \le \sum_{\bsk \in \cp_m^{\perp}\setminus \{\bszero\} } \abs{\hf(\bsk)}. \label{err1}
\end{align}

\section{The Fast Fourier Transform for Function Values at Rank-1 Lattice Node Sets}\label{FFT}

Adaptive Algorithm \ref{adapalgo} (\cublat) constructed in Section \ref{algorithmsection} has an error analysis based on the above expression.  However, the true Fourier coefficients are unknown and they must be approximated by the discrete coefficients, defined as:
\begin{subequations} \label{tfdefalias}
\begin{align}
\tf_m(\bsk)
&:= \hI_m\left( \E^{-2 \pi \sqrt{-1} \ip{\bsk}{\cdot}} f(\cdot) \right) \label{tfdef}\\
\nonumber
& = \hI_m\left( \E^{-2 \pi \sqrt{-1} \ip{\bsk}{\cdot}} \sum_{\bsl \in \Z^d} \hf(\bsl) \E^{2 \pi \sqrt{-1} \ip{\bsl}{\cdot}}  \right) \\
\nonumber
& = \sum_{\bsl \in \Z^d} \hf(\bsl) \hI_m\left( \E^{2 \pi \sqrt{-1} \ip{\bsl-\bsk}{\cdot}}  \right) \\
\nonumber
& = \sum_{\bsl \in \Z^d} \hf(\bsl) \E^{2 \pi \sqrt{-1} \ip{\bsl - \bsk}{\bsDelta}} \ind_{\cp_m^{\perp}}(\bsl - \bsk) \\
\nonumber
& = \sum_{\bsl \in \cp^{\perp}_m} \hf(\bsk+\bsl) \E^{2 \pi \sqrt{-1} \ip{\bsl }{\bsDelta}} \\
&= \hf(\bsk) + \sum_{\bsl \in \cp^{\perp}_m\setminus \{\bszero\}} \hf(\bsk+\bsl) \E^{2 \pi \sqrt{-1} \ip{\bsl }{\bsDelta}}, \qquad \forall \bsk \in \Z^d. \label{tfassum}
\end{align}
\end{subequations}
Thus, the discrete transform $\tf_m(\bsk)$ equals the integral transform $\hf(\bsk)$, defined in \eqref{Fourierdef}, plus \emph{aliasing} terms corresponding to $\hf(\bsk+\bsl)$ scaled by the shift, $\bsDelta$, where $\bsl \in \cp_{m}^{\perp}\setminus \left\{\bszero\right\}$.

To facilitate the calculation of $\tf_m(\bsk)$, we define the map $\tnu_m : \Z^d \to \F_{b^m}$ as follows:
\begin{equation} \label{kdotzbm}
\tnu_{0}(\bsk) := 0, \qquad
\tnu_{m}(\bsk) := b^{m} \ip{\bsk}{\bsz_{b^{m-1}}}, \quad m \in \N.
\end{equation}
A simple but useful remark is that $\cp^{\perp,j}_m$ corresponds to all $\bsk \in \Z^d$ such that $\tnu_{m}(\bsk)=j$ for $ j\in\F_{b^m}$. The above definition implies that $\ip{\bsk}{\bsz_i}$ appearing in $\tf_m(\bsk)$, may be written as
\begin{multline} \label{kdotzi}
\ip{\bsk}{\bsz_i} = \ip{\bsk}{\sum_{\ell=0}^{m-1} i_\ell \bsz_{b^{\ell}} \bmod 1} = \sum_{\ell=0}^{m-1} i_\ell \ip{\bsk}{\bsz_{b^{\ell}}} \bmod 1 \\
=\sum_{\ell=0}^{m-1} i_\ell \tnu_{\ell+1}(\bsk)  b^{-\ell-1} \bmod 1.
\end{multline}

The map $\tnu_m$ depends on the choice of the embedded rank-1 lattice node sets defined in \eqref{zbmdef} and \eqref{zidef}.  We can confirm that the right hand side of this definition lies in $\F_{b^m}$ by appealing to  \eqref{zbmdef} and recalling that the $\bsa_\ell$ are integer vectors:
\begin{align*}
b^{m} \ip{\bsk}{\bsz_{b^{m-1}}} &= b^m[(b^{-1} \bsk^T\bsa_{m-1} + \cdots + b^{-m}\bsk^T \bsa_{0}) \bmod 1]\\
&= (b^{m-1} \bsk^T\bsa_{m-1} + \cdots + \bsk^T \bsa_{0}) \bmod b^m \in \F_{b^m}, \ m \in \N.
\end{align*}

Moreover, note that for all $m\in \N$
\begin{align}
\label{tnuprop}
\nonumber
\tnu_{m+1}(\bsk) - \tnu_{m}(\bsk) & = b^{m+1} \ip{\bsk}{\bsz_{b^{m}}} - b^{m} \ip{\bsk}{\bsz_{b^{m-1}}} \\
\nonumber
&=b^{m}[ b\ip{\bsk}{\bsz_{b^{m}}} - \ip{\bsk}{\bsz_{b^{m-1}}}]\\
\nonumber
&=b^{m}[ a + \ip{\bsk}{b\bsz_{b^{m}} \bmod 1}  - \ip{\bsk}{\bsz_{b^{m-1}}}], \quad \text{for some }a \in \F_b \\
\nonumber
&=b^{m}[ a + \ip{\bsk}{\bsz_{b^{m-1}}}  - \ip{\bsk}{\bsz_{b^{m-1}}}], \quad  \text{by \eqref{latpropb}} \\
&=a b^{m} \quad \text{for some }a \in \F_b.
\end{align}
For all $\nu\in \N_0$ with proper $b$-ary expansion $\nu= \nu_0 + \nu_1 b + \cdots \in \N_0$, let $\onu_{m}$ denote the integer obtained by keeping only the first $m$ terms of its $b$-ary expansion, i.e., 
\begin{equation}
\label{onudef}
\onu_{m} := \nu_0 + \cdots + \nu_{m-1}b^{m-1} = [(b^{-m} \nu) \bmod 1] b^{m} \in \F_{b^{m}}
\end{equation}
The derivation in \eqref{tnuprop} means that if $\tnu_{m}(\bsk) = \nu \in \F_{b^m}$, then
\begin{equation}\label{nupropexp}
\tnu_\ell(\bsk)=\onu_{\ell}, \qquad \ell=1, \ldots, m.
\end{equation} 

Letting $y_i:=f(\bsz_{i}\oplus \bsDelta)$ for $i\in \N_0$ and considering \eqref{kdotzi}, the discrete Fourier transform defined in \eqref{tfdef} can now be written as follows:
\begin{align}
\nonumber
\tf_m(\bsk)
&:= \hI_m\left( \E^{-2 \pi \sqrt{-1} \ip{\bsk}{\cdot}} f(\cdot) \right) = \frac{1}{b^m}\sum_{i=0}^{b^m-1} \E^{-2 \pi \sqrt{-1} \ip{\bsk}{\bsz_{i}\oplus \bsDelta}} y_i \\
&= \E^{-2 \pi \sqrt{-1} \ip{\bsk}{\bsDelta}} Y_{m}(\tnu_m(\bsk)), \qquad m \in \N_0, \ \bsk \in \Z^d, \label{tfmYdef}
\intertext{where for all $m, \nu \in \N_0$,}
\nonumber
Y_m (\nu) & := \frac{1}{b^m} \sum_{i_{m-1}=0}^{b-1} \cdots \sum_{i_{0}=0}^{b-1} y_{i_0 + \cdots +i_{m-1} b^{m-1}} \exp\left(-2 \pi \sqrt{-1}\sum_{\ell=0}^{m-1} i_\ell \onu_{\ell+1}  b^{-\ell-1} \right)\\
\nonumber
&=Y_m(\onu_m).
\end{align}
The quantity $Y_m(\nu)$, $\nu \in \F_{b^m}$, which is essentially the discrete Fourier transform, can be computed efficiently via some intermediate quantities. For $p \in \{0, \ldots, m-1\}$, $m,\nu \in \N_0$ define $Y_{m,0}(i_{0}, \ldots, i_{m-1}) := y_{i_0 + \cdots +i_{m-1} b^{m-1}}$ and let
\begin{align*}
\MoveEqLeft{Y_{m,m-p}(\nu,i_{m-p}, \ldots, i_{m-1})} \\
& :=\frac{1}{b^{m-p}} \sum_{i_{m-p-1}=0}^{b-1} \cdots  \sum_{i_{0}=0}^{b-1} y_{i_0 + \cdots + i_{m-1} b^{m-1}} \exp\left( -2 \pi \sqrt{-1}\sum_{\ell=0}^{m-p-1} i_\ell \onu_{\ell+1}  b^{-\ell-1}  \right).
\end{align*}
Note that $Y_{m,m-p}(\nu,i_{m-p}, \ldots, i_{m-1})=Y_{m,m-p}(\onu_{m-p},i_{m-p}, \ldots, i_{m-1})$, and thus takes on only $b^m$ distinct values.  Also note that $Y_{m,m}(\nu)= Y_m(\nu)$. For $p=m-1, \ldots, 0$, compute
\begin{align*}
\MoveEqLeft{Y_{m,m-p}(\nu,i_{m-p}, \ldots, i_{m-1})} \\
& =\frac{1}{b^{m-p}} \sum_{i_{m-p-1}=0}^{b-1} \cdots  \sum_{i_{0}=0}^{b-1} y_{i_0 + \cdots + i_{m-1} b^{m-1}} \exp\left( -2 \pi \sqrt{-1} \sum_{\ell=0}^{m-p-1} i_\ell \onu_{\ell+1}  b^{-\ell-1}   \right)\\
& =\frac{1}{b} \sum_{i_{m-p-1}=0}^{b-1} Y_{m,m-p-1}(\nu,i_{m-p-1}, \ldots, i_{m-1})\exp\left( -2 \pi \sqrt{-1}  i_{m-p-1} \onu_{m-p} b^{-m+p}  \right).
\end{align*}
For each $p$ one must perform $\Order(b^m)$ operations, so the total computational cost to obtain  $Y_m(\nu)$ for all $\nu \in \F_{b^m}$ is $\Order(mb^m)$.

\section{Error Estimation}\label{secerrest}

As seen in equation \eqref{err1}, the absolute error is bounded by a sum of the absolute value of the Fourier coefficients in the dual lattice. Note that increasing the number of points in our lattice, i.e. increasing $m$, removes wavenumbers from the set over which this summation is defined. However, it is not obvious how fast is this error decreasing with respect to $m$. Rather than deal with a sum over the vector wavenumbers, it is more convenient to sum over scalar non-negative integers.  Thus, we define another mapping $\tvk: \N_0 \to \Z^d$.

\begin{definition} \label{wavenummapdef} Given a sequence of points in embedded lattices, $\cp_{\infty} = \{\bsz_i\}_{i=0}^{\infty}$ define $\tvk: \N_0 \to \Z^d$ \emph{one-to-one} and \emph{onto} recursively as follows:
\begin{tabbing}
\hspace{0.5cm} \= Set $\tvk(0)=\bszero$ \+ \\
For $m\in \N_0$ \\
\hspace{0.3cm} \= For $\kappa \in \F_{b^m}$ ,  \+ \\
\hspace{0.3cm} \= Let $a\in \F_b$ be such that $\tnu_{m+1}(\tvk(\kappa))= \tnu_{m}(\tvk(\kappa)) + ab^m$. \+ \\
i) \hspace{1ex} \= If $a\ne 0$, choose $\tvk(\kappa+a b^m) \in \{\bsk \in  \Z^d : \tnu_{m+1}(\bsk)=\tnu_{m}(\tvk(\kappa))\}$. \\
ii) \> Choose $\tvk(\kappa+a' b^m) \in \{\bsk \in  \Z^d : \tnu_{m+1}(\bsk)=\tnu_{m}(\tvk(\kappa))+a' b^m \}$, \\ \` for  $a'\in \{1, \ldots, b-1\}\setminus \{a\}$.
\end{tabbing}
\end{definition}

Definition \ref{wavenummapdef} is intended to reflect the embedding of the \emph{dual} cosets described in \eqref{dualemb} and \eqref{dualcosetemb}. For clarity, consider $\tnu_{m}(\tvk(\kappa))=j$. In i), if $\tvk(\kappa)\in\cp^{\perp,j+ab^m}_{m+1}$ with $a>0$, we choose $\tvk(\kappa+ab^m)\in\cp^{\perp,j}_{m+1}$. Otherwise by ii), we simply choose $\tvk(\kappa+a' b^m)\in\cp^{\perp,j+a' b^m}_{m+1}$. Condition i) forces us to pick wavenumbers in $\cp^{\perp,j}_{m+1}$.

This mapping is not uniquely defined and one has the flexibility to choose part of it. For example, defining a norm such as in \cite[Chap. 4]{SloJoe94} one can assign smaller values of $\kappa$ to smaller wavenumbers $\bsk$. In the end, our goal is to define this mapping such that $\hf(\tvk(\kappa))\rightarrow 0$ as $\kappa \to \infty$. In addition, it is one-to-one since at each step the new values $\tvk(\kappa+a b^m)$ or $\tvk(\kappa+a' b^m)$ are chosen from sets of wavenumbers that exclude those wavenumbers already assigned to $\tvk(\kappa)$.  The mapping can be made onto by choosing the ``smallest'' wavenumber in some sense. 

It remains to be shown that for any $\kappa \in \F_{b^{m}}$, $\{\bsk \in  \Z^d : \tnu_{m+1}(\bsk)=\tnu_{m}(\tvk(\kappa))+a' b^m\}$ is nonempty for all $a' \in \F_b$ with $a' \ne a$.  Choose $\bsl$ such that $\ip{\bsl}{\bsz_1}=b^{-1}$. This is possible because $\bsz_1 = b^{-1} \bsa_0 \ne \bszero$.  For any $m \in \N_0$, $\kappa \in \F_{b^m}$, and  $a'' \in \F_b$, note that
\begin{align*}
\ip{\tvk(\kappa) + a'' b^{m} \bsl}{\bsz_{b^{m}}}& = \ip{\tvk(\kappa)}{\bsz_{b^{m}}} + a'' b^m \ip{\bsl}{\bsz_{b^{m}}} \bmod 1 \qquad \qquad \text{by \eqref{bilinearlinkprop}} \\
&= [b^{-m-1}\tnu_{m+1}(\tvk(\kappa)) + a'' \ip{\bsl}{b^m  \bsz_{b^{m}}\bmod 1}] \bmod 1 \\
& \qquad \qquad \qquad  \text{by \eqref{bilinearlinxprop} and \eqref{kdotzbm}}\\
&= [b^{-m-1}\tnu_{m}(\tvk(\kappa)) + a b^{-1} + a'' \ip{\bsl}{\bsz_{1}}] \bmod 1 \qquad  \text{by \eqref{latpropb}}\\
&= [b^{-m-1}\tnu_{m}(\tvk(\kappa)) + (a+a'') b^{-1}] \bmod 1,
\end{align*}
Then it follows that
\begin{equation*}
\tnu_{m+1}(\tvk(\kappa) + a'' b^{m} \bsl) = 
\tnu_{m}(\tvk(\kappa)) + (a+a'' \bmod b) b^{m} \qquad \text{by \eqref{kdotzbm}}.
\end{equation*}
By choosing $a''$ such that $a'=(a+a'' \bmod b)$, we have shown that the set $\kappa \in \F_{b^{m}}$, $\{\bsk \in  \Z^d : \tnu_{m+1}(\bsk)=\tnu_{m}(\tvk(\kappa))+a' b^m\}$ is nonempty.

To illustrate the initial steps of a possible mapping, consider the lattice in Figure \ref{Latticefig} and Table \ref{tablemap}. For $m=0$, $\kappa\in\{0\}$ and $a=0$. This skips i) and implies $\tvk(1)\in\{\bsk \in  \Z^d :\tnu_{1}(\bsk)=2
\ip{\bsk}{(1,27)/2}=1\}$, so one may choose $\tvk(1):=(-1,0)$. After that, $m=1$ and $\kappa\in\{0,1\}$. Starting with $\kappa=0$, again $a=0$ and we jump to ii) where we require $\tvk(2)\in\{\bsk \in  \Z^d : \tnu_{2}(\bsk)=4\ip{\bsk}{(1,27)/4}=2\}$ and thus, we can take $\tvk(2):=(-1,1)$. When $\kappa=1$, we note that $\tnu_{2}(\tvk(1))=\tnu((-1,0))=3$. Here $a=1$ leading to i) and $\tvk(3)\in\{\bsk \in  \Z^d : \tnu_{2}(\bsk)=1\}$, so we may choose $\tvk(3):=(1,0)$. Continuing, we may take $\tvk(4):=(-1,-1)$, $\tvk(5):=(0,1)$, $\tvk(6):=(1,-1)$ and $\tvk(7):=(0,-1)$.

\begin{table}[h]
\begin{center}
\begin{tabular}{{c|}*{3}{|c|}{|c}}
$\tvk(\kappa)$ & $\kappa$ & \vtop{\hbox{$\tnu_{1}(\tvk(\kappa))=$}\hbox{\hspace{0.6cm}$2
\ip{\tvk(\kappa)}{(1,27)/2}$}} & \vtop{\hbox{$\tnu_{2}(\tvk(\kappa))=$}\hbox{\hspace{0.6cm}$4\ip{\tvk(\kappa)}{(1,27)/4}$}} & \vtop{\hbox{$\tnu_{3}(\tvk(\kappa))=$}\hbox{\hspace{0.6cm}$8\ip{\tvk(\kappa)}{(1,27)/8}$}} \\
\hline
$(0,0)$ & 0 & 0 & 0 & 0   \\
$(-1,-1)$ & 4 & 0 & 0 & 4   \\
$(-1,1)$ & 2 & 0 & 2 & 2   \\
$(1,-1)$ & 6 & 0 & 2 & 6   \\
$(-1,0)$ & 1 & 1 & 3 & 7   \\
$(1,0)$ & 3 & 1 & 1 & 1   \\
$(0,-1)$ & 7 & 1 & 1 & 5   \\
$(0,1)$ & 5 & 1 & 3 & 3   \\
$(1,1)$ & $\cdots$ & 0 & 0 & 4   \\
\hline
\end{tabular}
\caption{Values $\tnu_{1}$, $\tnu_{2}$ and $\tnu_{3}$ for some wavenumbers and a possible assignment of $\tvk(\kappa)$. The reader should notice that $\tnu_{m+1}(\tvk(\kappa))-\tnu_{m}(\tvk(\kappa))$ is either $0$ or $2^{m}$.}\label{tablemap}
\end{center}
\end{table}
\vspace{-.5cm}
\begin{lemma} \label{tvklemma}
The map in Definition  \ref{wavenummapdef} has the property that for $m \in \N_0$ and $\kappa \in \F_{b^m}$,
\[
\{\tvk(\kappa + \lambda b^m)\}_{\lambda=0}^{\infty} =\{\bsl \in \Z^d : \tvk(\kappa) - \bsl \in \cp_m^{\perp} \}.
\]
\end{lemma}
\begin{proof} This statement holds trivially for $m=0$ and $\kappa=0$.  For $m \in \N$ it is noted that
\begin{align}
\nonumber
\bsk-\bsl \in \cp_m^{\perp} & \iff \ip{\bsk-\bsl}{\bsz_{b^{m-1}}}=0 \qquad \text{by \eqref{dualdef}} \\
\nonumber
& \iff \ip{\bsk}{\bsz_{b^{m-1}}}=\ip{\bsl}{\bsz_{b^{m-1}}} \qquad \text{by \eqref{bilinearlinkprop}}\\
\nonumber
& \iff b^{-m}\tnu_m(\bsk) =b^{-m}\tnu_m(\bsl) \qquad \text{by \eqref{kdotzbm}} \\
& \iff \tnu_m(\bsk) =\tnu_m(\bsl). \label{kminlinPperp}
\end{align}
This implies that for all $m \in \N$ and $\kappa \in \F_{b^m}$,
\begin{equation}
\{\bsl \in \Z^d : \tnu_m(\bsl) =  \tnu_m(\tvk(\kappa))\} = \{\bsl \in \Z^d : \tvk(\kappa) - \bsl \in \cp_m^{\perp} \}.
\end{equation}

By Definition \ref{wavenummapdef} it follows that for $m \in \N$ and $\kappa \in \F_{b^m}$,
\begin{align}
\nonumber
\{\tvk(\kappa +\lambda b^m)\}_{\lambda=0}^{b-1} 
& \subseteq \{\bsk \in  \Z^d : \tnu_{m+1}(\bsk)=\tnu_m(\tvk(\kappa))+a b^m, \ a \in \F_b \} \\
\nonumber
& = \{\bsk \in  \Z^d : \tnu_{m}(\bsk)=\tnu_m(\tvk(\kappa)) \}. 
\intertext{Applying property \eqref{nupropexp} on the right side,}
\nonumber
\{\tvk(\kappa +\lambda b^m)\}_{\lambda=0}^{b-1}& \subseteq \{\bsk \in  \Z^d : \tnu_{\ell}(\bsk)=\tnu_\ell(\tvk(\okappa_\ell)) \}, \qquad \forall \ell=1, \ldots, m.
\intertext{Because one can say the above equation holds $\forall\ell=1,\dots,n<m$, the left hand side can be extended,} 
\{\tvk(\kappa +\lambda b^m)\}_{\lambda=0}^{\infty}& \subseteq \{\bsk \in  \Z^d : \tnu_{m}(\bsk)=\tnu_m(\tvk(\kappa)) \}. \label{oneway}
\end{align}

Now suppose that $\bsl$ is any element of $\{\bsk \in  \Z^d : \tnu_{m}(\bsk)=\tnu_m(\tvk(\kappa)) \}$.  Since the map $\tvk$ is onto, there exists some $\kappa' \in \N_0$ such that $\bsl=\tvk(\kappa')$. Choose $\lambda'$ such that $\kappa'=\overline{\kappa'}_m +\lambda' b^m$, where the overbar notation was defined in \eqref{onudef}.  According to \eqref{oneway} it follows that $\tnu_m(\tvk(\overline{\kappa'}_m))=\tnu_m(\tvk(\overline{\kappa'}_m +\lambda' b^m)) = \tnu_m(\bsl)=\tnu_m(\tvk(\kappa))$.  Since $\overline{\kappa'}_m$ and $\kappa$ are both in $\F_{b^m}$, this implies that $\overline{\kappa'}_m=\kappa$, and so $\bsl \in \{\tvk(\kappa +\lambda b^m)\}_{\lambda=0}^{\infty}$.  Thus, $\{\tvk(\kappa +\lambda b^m)\}_{\lambda=0}^{\infty} \supseteq \{\bsk \in  \Z^d : \tnu_{m}(\bsk)=\tnu_m(\tvk(\kappa)) \}$, and the lemma is proved. \qed
\end{proof}

For convenience we adopt the notation $\hf_{\kappa} :=\hf(\tvk(\kappa))$ and $\tf_{m,\kappa} := \tf_m(\tvk(\kappa))$. Then, by Lemma  \ref{tvklemma} the error bound in \eqref{err1} may be written as
\begin{equation}
\biggabs{ \int_{\cube} f(\bsx) \, \D \bsx - \hI_m(f)} 
\le \sum_{\lambda=1}^{\infty} \left \lvert \hf_{\lambda b^m}\right \rvert, \label{err2}
\end{equation}
and the aliasing relationship in \eqref{tfassum} becomes
\begin{align}
\tf_{m,\kappa} =\hf_{\kappa} + \sum_{\lambda=1}^{\infty} \hf_{\kappa+\lambda b^{m}} \E^{2 \pi \sqrt{-1} \ip{\tvk(\kappa+\lambda b^{m}) - \tvk(\kappa)}{\bsDelta}}. \label{tfassumc}
\end{align}

Given an integrand with absolutely summable Fourier coefficients, consider the following sums defined for $\ell,m \in \N_0$, $\ell \le m$:
\begin{gather*}
S_m(f) =  \sum_{\kappa=\left \lfloor b^{m-1} \right \rfloor}^{b^{m}-1} \bigabs{\hf_{\kappa}}, \qquad 
\hS_{\ell,m}(f)  = \sum_{\kappa=\left \lfloor b^{\ell-1} \right \rfloor}^{b^{\ell}-1} \sum_{\lambda=1}^{\infty} \bigabs{ \hf_{\kappa+\lambda b^{m}}}, \\
\wcS_m(f)=\hS_{0,m}(f) + \cdots + \hS_{m,m}(f)=
\sum_{\kappa=b^{m}}^{\infty} \bigabs{\hf_{\kappa}}, \qquad
\tS_{\ell,m}(f) = \sum_{\kappa=\left \lfloor b^{\ell-1}\right \rfloor}^{b^{\ell}-1} \bigabs{\tf_{m,\kappa}}.
\end{gather*}
Note that $\tS_{\ell,m}(f)$ is the only one that can be observed from data because it involves the discrete transform coefficients. In fact, from \eqref{tfmYdef} one can identify $\bigabs{\tf_{m,\kappa}} = \bigabs{Y_m(\tnu_m(\tvk(\kappa)))}$ and our adaptive algorithm will be based on this sum bounding the other three, $S_m(f),\hS_{\ell,m}(f)$, and $\wcS_m(f)$, which cannot be readily observed. 

Let  $\ell_* \in \N$ be some fixed integer and $\homega$ and $\wcomega$ be some bounded non-negative valued functions.  We define a \emph{cone}, $\cc$, of absolutely continuous functions whose Fourier coefficients decay according to certain inequalities:
\begin{multline} \label{conecond}
\cc:=\{f \in AC(\cube) : \hS_{\ell,m}(f) \le \homega(m-\ell) \wcS_m(f),\ \ \ell \le m, \\
\wcS_m(f) \le \wcomega(m-\ell) S_{\ell}(f),\ \  \ell_* \le \ell \le m\}.
\end{multline}
We also require the existence of $r$ such that $\homega(r) \wcomega(r)<1$ and that $\lim_{m \to \infty} \wcomega(m) = 0$. This set is a cone, i.e. $f \in \cc \implies af \in \cc\;\forall a\in\R$, but it is not convex. A wider discussion on the advantages and disadvantages of designing numerical algorithms for cones of functions can be found in \cite{Clancy201421}.

Functions in $\cc$ have Fourier coefficients that do not oscillate wildly. According to \eqref{err2}, the error of our integration is bounded by $\hS_{0,m}(f)$. Nevertheless, for practical purposes we will use $S_{\ell}(f)$ as an indicator for the error. Intuitively, the cone conditions enforce these two sums to follow a similar trend. Thus, one can expect that small values of $S_{\ell}(f)$ imply small values of $\hS_{0,m}(f)$.

The first inequality controls how an infinite sum of \emph{some} of the larger wavenumber coefficients are bounded above by a sum of all the surrounding coefficients. The second inequality controls how the sum of these surrounding coefficients is bounded above by a \emph{finite} sum of some smaller wavenumber Fourier coefficients. In Figure \ref{Walshcoeffig} we can see how $S_8(f)$ can be used to bound $\wcS_{12}(f)$ and $\wcS_{12}(f)$ to bound $\hS_{0,12}(f)$. The former sum also corresponds to the error bound in \eqref{err2}.

For small $\ell$ the sum $S_\ell(f)$ includes only a few summands. Therefore, it could accidentally happen that $S_\ell(f)$ is too small compared to $\wcS_m(f)$. To avoid this possibility, the cone definition includes the constraint that $\ell$ is greater than some minimum $\ell_*$.

\begin{figure}
\centering
\includegraphics[width=8.0cm]{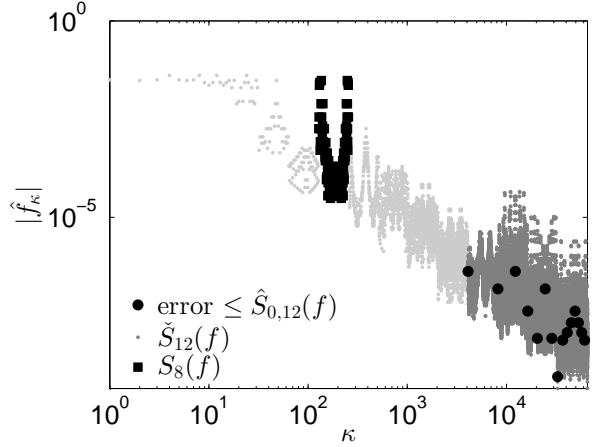}
\caption{The magnitudes of true Fourier coefficients for some integrand. \label{Walshcoeffig}}
\end{figure}

Because we do not assume the knowledge of the true  Fourier  coefficients, for functions in $\cc$ we need bounds on $S_{\ell}(f)$ in terms of the sum of the discrete coefficients $\tS_{\ell,m}(f)$.  This is done by applying \eqref{tfassumc},  and the definition of the cone in \eqref{conecond}:
\begin{align}
\nonumber
S_\ell(f) &= \sum_{\kappa=\left \lfloor b^{\ell-1} \right \rfloor}^{b^{\ell}-1} \bigl \lvert \hf_{\kappa}\bigr\rvert= \sum_{\kappa=\left \lfloor b^{\ell-1} \right \rfloor}^{b^{\ell}-1} \abs{\tf_{m,\kappa} - \sum_{\lambda=1}^{\infty} \hf_{\kappa+\lambda b^{m}} \E^{2 \pi \sqrt{-1} \ip{\tvk(\kappa+\lambda b^{m}) - \tvk(\kappa)}{\bsDelta}}}\\
\nonumber
&\le \sum_{\kappa=\left \lfloor b^{\ell-1} \right \rfloor}^{b^{\ell}-1} \bigl \lvert \tf_{m,\kappa} \bigr\rvert + \sum_{\kappa=\left \lfloor b^{\ell-1} \right \rfloor}^{b^{\ell}-1} \sum_{\lambda=1}^{\infty} \bigl \lvert \hf_{\kappa+\lambda b^{m}}\bigr\rvert = \tS_{\ell,m}(f) + \hS_{\ell,m}(f) \\
\label{boundSumbsApproxa}
&\le \tS_{\ell,m}(f) + \homega(m-\ell) \wcomega(m-\ell) S_\ell(f)\
\end{align}
and provided that $\homega(m-\ell) \wcomega(m-\ell)<1$,
\begin{equation}\label{boundSumsApproxb}
S_\ell(f) \le \frac{\tS_{\ell,m}(f)}{1 - \homega(m-\ell) \wcomega(m-\ell)}.
\end{equation}
By \eqref{err2} and the cone conditions, \eqref{boundSumsApproxb} implies a data-based error bound:
\begin{align}
\nonumber
\biggabs{\int_{\cube} f(\bsx) \, \D \bsx - \hI_m(f) }
&\le \sum_{\lambda=1}^{\infty} \bigabs{\hf_{\lambda b^{m}}} 
= \hS_{0,m}(f)\le \homega(m) \wcS_m(f)\\
\nonumber
&  \le \homega(m) \wcomega(m-\ell) S_\ell(f)\\
& \le  \frac{\homega(m) \wcomega(m-\ell)}{1 - \homega(m-\ell) \wcomega(m-\ell)}\tS_{\ell,m}(f).
\label{SSbd2}
\end{align}
In Section \ref{algorithmsection} we construct an adaptive algorithm based on this conservative bound. 

\section{An Adaptive Algorithm Based for Cones of Integrads}\label{algorithmsection}

Inequality \eqref{SSbd2} suggests the following algorithm. First, choose $\ell_*$ and fix $r:=m-\ell \in \N$ such that $\homega(r)\wcomega(r)<1$ for $\ell\geq\ell_*$. Then, define
\[
\fC(m):= \frac{\homega(m) \wcomega(r)}{1 - \homega(r) \wcomega(r)}.
\]

The choice of the parameter $r$ is important. Larger $r$ means a smaller $\fC(m)$, but it also makes the error bound more dependent on smaller indexed Fourier coefficients.

\begin{algo}[Adaptive Rank-1 Lattice Cubature, \cublat] \label{adapalgo} Fix $r$ and $\ell_*$, $\homega$ and $\wcomega$ describing $\cc$ in \eqref{conecond}. Given a tolerance, $\varepsilon$, initialize $m=\ell_*+r$ and do:

\begin{description}
\item[\textbf{Step 1.}] According to Section \ref{FFT}, compute $\tS_{m-r,m}(f)$.
\item[\textbf{Step 2.}] Check whether $\fC(m)  \tS_{m-r,m}(f) \le \varepsilon$. If true, return $\hI_m(f)$ defined in \eqref{cubaturedef}. If not, increment $m$ by one, and go to Step 1.
\end{description}
\end{algo}

\begin{theorem} \label{adapalgothm} For $m = \min \{m' \ge \ell_*+r : \fC(m')  \tS_{m'-r,m'}(f) \le \varepsilon \}$, Algorithm \ref{adapalgo} is successful whenever $f\in\cc$,
\[
\biggabs{\int_{\cube} f(\bsx) \D \bsx - \hI_m(f)} \le \varepsilon.
\]
Thus, the number of function data needed is $b^m$. Defining $m^* = \min \{m' \ge \ell_*+r : \fC(m') [1+ \homega(r) \wcomega(r)] S_{m'-r}(f) \le \varepsilon \}$, we also have $b^m\leq b^{m^*}$. This means that the computational cost can be bounded,
\[
\mathrm{cost}\left(\widehat{I}_m,f,\varepsilon\right)\leq \$(f)b^{m^*}+cm^*b^{m^*}
\]
where $\$(f)$ is the cost of evaluating $f$ at one data point.
\end{theorem}

\begin{proof}
By construction, the algorithm must be successful. Recall that the inequality used for building the algorithm is \eqref{SSbd2}.

To find the upper bound on the computational cost, a similar result to \eqref{boundSumbsApproxa} provides
\begin{align}
\nonumber
\tS_{\ell,m}(f) &= \sum_{\kappa=b^{\ell-1}}^{b^{\ell}-1} \bigabs{ \tf_{m,\kappa}} = \sum_{\kappa=b^{\ell-1}}^{b^{\ell}-1} \biggabs{\hf_{\kappa} + \sum_{\lambda=1}^{\infty} \hf_{\kappa+\lambda b^{m}} \E^{2 \pi \sqrt{-1} \ip{\tvk(\kappa+\lambda b^{m}) - \tvk(\kappa)}{\bsDelta}}}\\
\nonumber
&\le \sum_{\kappa=b^{\ell-1}}^{b^{\ell}-1} \bigabs{\hf_{\kappa}} + \sum_{\kappa=b^{\ell-1}}^{b^{\ell}-1} \sum_{\lambda=1}^{\infty} \bigabs{\hf_{\kappa+\lambda b^{m}}} 
= S_{\ell}(f) + \hS_{\ell,m}(f) \\
\nonumber
&\le [1  + \homega(m-\ell) \wcomega(m-\ell)] S_\ell(f). \label{SSbd3}
\end{align}
Replacing $\tS_{\ell,m}(f)$ in the error bound in \eqref{SSbd2} by the right hand side above proves that the choice of $m$ needed to satisfy the tolerance is no greater than  $m^*$ defined above.

In Section \ref{FFT}, the computation of $\tS_{m-r,m}(f)$ is described in terms of $\Order(mb^m)$ operations. Thus, the total cost of Algorithm \ref{adapalgo} is,
\[
\mathrm{cost}\left(\widehat{I}_m,f,\varepsilon\right)\leq \$(f)b^{m^*}+cm^*b^{m^*}
\]
\hfill \qed
\end{proof}

\section{Numerical Example} \label{secnumexpsec}

Algorithm \ref{adapalgo} has been coded in MATLAB as \cublat in base 2, and is part of GAIL, \cite{ChoEtal15a}. To test it, we priced an Asian call with geometric Brownian motion, $S_0=K=100$, $T=1$ and $r=3\%$. The test is performed on 500 samples whose dimensions are chosen IID uniformly among $1, 2, 4, 8, 16, 32,$ and $64$, and the volatility also IID uniformly from $10\%$ to $70\%$. Results, in Figure \ref{geoAsianmean}, show $97\%$ of success meeting the error tolerance.

The algorithm cone parametrization was $\ell_*=6$, $r=4$ and $\fC(m)=5 \times 2^{-m}$. In addition, each replication used a shifted lattice with $\bsDelta\sim U(0,1)$. However, results are strongly dependent on the generating vector that was used for creating the rank-1 lattice embedded node sets. The vector applied to this example was found with the \texttt{latbuilder} software from Pierre L'Ecuyer and David Munger \cite{LEcuyer:2015:AGS}, obtained for $2^{26}$ points, $d=250$ and coordinate weights $\gamma_j=j^{-2}$, optimizing the $P_2$ criterion.

For this particular example, the choice of $\fC(m)$ does not have a noticeable impact on the success rate or execution time. In other cases such as discontinuous functions, it is more sensitive. Being an adaptive algorithm, if the Fourier coefficients decrease quickly, cone conditions have a weaker effect. One can see that the number of summands involving $\tS_{m-r,m}(f)$ is $2^{m-r-1}$ for a fixed $r$. Thus, in order to give a uniform weight to each wavenumber, we chose $\fC(m)$ proportional to $2^{-m}$.

\begin{figure}[h!]
\centering
\includegraphics[width=7.5cm]{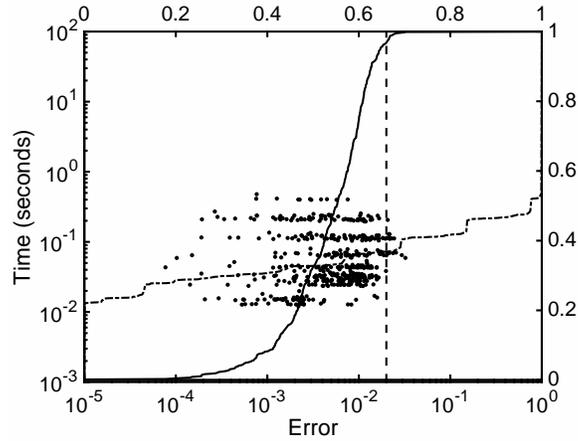} 
\caption{Empirical distribution functions obtained from 500 samples, for the error (continuous line) and time (slashed-doted line). Quantiles are specified on the right and top axes respectively. The tolerance of 0.02 (vertical dashed line) is an input of the algorithm and will be a guaranteed bound on the error if the function lies inside the cone. \label{geoAsianmean}}
\end{figure}

\section{Discussion and Future Work}
Quasi-Monte Carlo methods rarely provide guaranteed adaptive algorithms. This new methodology that bounds the absolute error via the discrete Fourier coefficients allows us to build an adaptive automatic algorithm guaranteed for cones of integrands. The non-convexity of the cone allows our adaptive, nonlinear algorithm to be advantageous in comparison with non-adaptive, linear algorithms.  

Unfortunately, the definition of the cone does contain parameters, $\homega$ and $\wcomega$, whose optimal values may be hard to determine.  Moreover, the definition of the cone does not yet correspond to traditional sets of integrands, such as  Korobov spaces.  These topics deserve further research.

Concerning the generating vector used in Section \ref{secnumexpsec}, some further research should be carried out to specify the connection between dimension weights and cone parameters. This might lead to the existence of optimal weights and generating vector.

Our algorithm provides an upper bound on the complexity of the problem, but we have not yet obtained a lower bound. We are also interested in extending our algorithm to accommodate a relative error tolerance.   We would like to understand how the cone parameters might depend on the dimension of the problem, and we would like to extend our adaptive algorithm to infinite dimensional problems via multi-level or multivariate decomposition methods.

\begin{acknowledgement}
The authors thank Ronald Cools and Dirk Nuyens for organizing MCQMC 2014 and greatly appreciate the suggestions made by Sou-Cheng Choi, Frances Kuo, Lan Jiang, Dirk Nuyens and Yizhi Zhang to improve this manuscript. In addition, the first author also thanks Art B. Owen for partially funding traveling expenses to MCQMC 2014 through the US National Science Foundation (NSF).
This work was partially supported by NSF grants DMS-1115392,  DMS-1357690, and DMS-1522687. 
\end{acknowledgement}


\begin{thebibliography}{10}
\providecommand{\url}[1]{{#1}}
\providecommand{\urlprefix}{URL }
\expandafter\ifx\csname urlstyle\endcsname\relax
  \providecommand{\doi}[1]{DOI~\discretionary{}{}{}#1}\else
  \providecommand{\doi}{DOI~\discretionary{}{}{}\begingroup
  \urlstyle{rm}\Url}\fi

\bibitem{ChoEtal15a}
Choi, S.C.T., Ding, Y., Hickernell, F.J., Jiang, L., {Jim\'enez Rugama},
  {\relax Ll}.A., Tong, X., Zhang, Y., Zhou, X.: {GAIL}: {G}uaranteed
  {A}utomatic {I}ntegration {L}ibrary (versions 1.0--2.1).
\newblock MATLAB software ({2013--2015}).
\newblock \urlprefix\url{https://github.com/GailGithub/GAIL_Dev}.

\bibitem{Clancy201421}
Clancy, N., Ding, Y., Hamilton, C., Hickernell, F.J., Zhang, Y.: The cost of
  deterministic, adaptive, automatic algorithms: Cones, not balls.
\newblock Journal of Complexity \textbf{30}(1), 21 -- 45 (2014).

\bibitem{DicEtal14a}
Dick, J., Kuo, F., Sloan, I.H.: High dimensional integration --- the
  {Q}uasi-{M}onte {C}arlo way.
\newblock Acta Numerica \textbf{22}, 133--288 (2013).

\bibitem{Hic01a}
Hickernell, F.J.: Obtaining {$O(N^{-2+\epsilon})$} convergence for lattice
  quadrature rules.
\newblock In: K.T. Fang, F.J. Hickernell, H.~Niederreiter (eds.) {M}onte
  {C}arlo and Quasi-{M}onte {C}arlo Methods 2000, pp. 274--289.
  Springer-Verlag, Berlin (2002).

\bibitem{HicJim16a}
Hickernell, F.J., {Jim\'enez Rugama}, {\relax Ll}.A.: Reliable adaptive
  cubature using digital sequences.
\newblock In: R.~Cools, D.~Nuyens (eds.) {M}onte {C}arlo and Quasi-{M}onte
  {C}arlo Methods 2014. Springer-Verlag, Berlin (2015+).
\newblock To appear, arXiv:1410.8615.

\bibitem{HicNie03a}
Hickernell, F.J., Niederreiter, H.: The existence of good extensible rank-1
  lattices.
\newblock Journal of Complexity \textbf{19}, 286--300 (2003).

\bibitem{HicSloWas03a}
Hickernell, F.J., Sloan, I.H., Wasilkowski, G.W.: On tractability of
  weighted integration over bounded and unbounded regions in {$\mathbb{R}^s$}.
\newblock Mathematics of Computation \textbf{73}, 1885--1901 (2004).

\bibitem{HicSloWas03e}
Hickernell, F.J., Sloan, I.H., Wasilkowski, G.W.: The strong tractability of
  multivariate integration using lattice rules.
\newblock In: H.~Niederreiter (ed.) {M}onte {C}arlo and Quasi-{M}onte {C}arlo
  Methods 2002, pp. 259--273. Springer-Verlag, Berlin (2004).

\bibitem{LEcuyer:2015:AGS}
L'Ecuyer, P., Munger, D.: Algorithm xxx: A general software tool for
  constructing rank-1 lattice rules.
\newblock {ACM} Transactions on Mathematical Software  (2015+).

\bibitem{Nie92}
Niederreiter, H.: Random Number Generation and Quasi-{M}onte {C}arlo Methods.
\newblock CBMS-NSF Regional Conference Series in Applied Mathematics. SIAM,
  Philadelphia (1992).

\bibitem{NovWoz10a}
Novak, E., Wo{\'{z}}niakowski, H.: Tractability of Multivariate Problems
  {V}olume II: {S}tandard Information for Functionals.
\newblock No.~12 in EMS Tracts in Mathematics. European Mathematical Society,
  Z\"urich (2010).

\bibitem{Sid93}
Sidi, A.: A new variable transformation for numerical integration.
\newblock In: H.~Brass, G.~H\"ammerlin (eds.) Numerical Integration IV, no. 112
  in International Series of Numerical Mathematics, pp. 359--373. Birkh\"auser,
  Basel (1993).

\bibitem{SloJoe94}
Sloan, I.H., Joe, S.: Lattice Methods for Multiple Integration.
\newblock Oxford University Press, Oxford (1994).

\end{thebibliography}

\end{document}